\documentclass[a4paper,12pt]{amsart}
\usepackage{amssymb,amsmath,amsfonts,amsthm}

\usepackage[margin=1.2in]{geometry}

\newcommand{\NN}{\mathbb N}

\newcommand{\RR}{\mathbb R}

\newcommand{\DD}{\mathcal D}

\theoremstyle{plain}
\newtheorem{theorem}{Theorem}[section]

\theoremstyle{remark}

\theoremstyle{definition}

\allowdisplaybreaks

\newcommand{\beq}{\begin{eqnarray}}
\newcommand{\eeq}{\end{eqnarray}}

\newcommand{\beqs}{\begin{eqnarray*}}
\newcommand{\eeqs}{\end{eqnarray*}}

\author[S. Pilipovi\' c]{Stevan Pilipovi\' c}
\thanks{The work of S. Pilipovi\'c was supported by the
grants of the
Ministry of Education, Science and
Technological development of
Republic of Serbia
174024.}
\address{Department of Mathematics and Informatics,
University of Novi Sad, Trg Dositeja Obradovi\'{c}a 4, 21000 Novi Sad, Serbia}
\email{stevan.pilipovic@dmi.uns.ac.rs}

\author[B. Prangoski]{Bojan Prangoski}
\thanks{The work of B. Prangoski was
partially supported by the COST Action CA15225
``Fractional-- order
systems--analysis, synthesis and
their importance for future design''
and by the
bilateral project
``Microlocal analysis and applications''
between the Macedonian and
Serbian academies of sciences and arts.}
\address{Department for Mathematics, Faculty of Mechanical
Engineering-Skopje, Karposh 2 b.b., 1000 Skopje, Macedonia}
\email{bprangoski@yahoo.com}

\title[On the characterisations of wave front sets via the STFT]{On the characterisations of wave front sets via the short-time Fourier transform}



\begin{document}

\begin{abstract} It is well known that  the classical  and Sobolev wave fronts were extended into non-equivalent global versions by the use of the short-time Fourier transform. In this very short paper we give  complete characterisations of initial wave front sets via the short-time Fourier transform.
\end{abstract}

\maketitle

The goal of this note is to give
descriptions of the classical (local) wave
front set as well as the Sobolev wave front set,
both in the sense of
H\" ormander \cite{hor11}--\cite{hor0}, of
a distribution $f$, by its short-time
Fourier transform, from now
on abbreviated as STFT. The STFT is also
known as the wave-packet
transform and it was introduced by
C\'ordoba and Fefferman \cite{corff}
 (see also \cite{groch}). The STFT of
 $f\in\mathcal{S}'(\mathbb{R}^d)$
 with a window (also known as wave packet)
  $0\neq\chi\in\mathcal{S}(\mathbb{R}^d)$ is defined
  by $V_{\chi}f(x,\xi)=\mathcal{F}_{t\rightarrow \xi}
  (f(t)\overline{\chi(t-x)})$,
  where $\mathcal{F}$ is the Fourier transform given by
  $\mathcal{F}g(\xi)=\int_{\mathbb{R}^d}e^{-ix\xi}g(x)dx$,
  $g\in L^1(\mathbb{R}^d)$. In fact, $V_{\chi}f$
  is a smooth function
  in $(x,\xi)$. When the window $\chi$ is
  compactly supported, we can
  naturally extend the definition of the STFT even when
  $f\in\mathcal{D}'(\mathbb{R}^d)$; namely
  $V_{\chi}f(x,\xi)=\langle e^{-i\xi \cdot} f,
  \overline{\chi(\cdot-x)}\rangle$.
  Clearly, even in this case,
  $V_{\chi}f\in\mathcal{C}^{\infty}(\mathbb{R}^{2d})$.\\
Folland \cite{fol} gave a characterisation of
the wave front set of $f\in\mathcal{S}'(\mathbb{R}^d)$
via its STFT under some restrictions on the window
$\chi\in\mathcal{S}(\mathbb{R}^d)$. Later
\=Okaji \cite{okaji} relaxed the restrictions
on the window and only recently
Kato, Kobayashi and Ito \cite{kkisf}
managed to gave a characterisation
without any restriction on $\chi$.
Our paper is motivated by \cite{kkisf}
where the authors have given a nice application
to the Schr\" odinger equation.

In several recent papers it was shown
that the homogeneous wave front
(cf.  \cite{martinez}, \cite{mns}, \cite{melrose},
\cite{mizahura}, \cite{nakamura}, \cite{RZ})
and the equivalent to it, Gabor wave front
set (cf. \cite{rw}, \cite{mcrs}, \cite{coriascosulc},
\cite{ScW}, \cite{ScW2}) are equivalent to the global
one of H\" ormander. Up to our knowledge, the local
(classical) $\mathcal{C}^{\infty}$ and Sobolev-type
H\" ormander's definitions of wave fronts based on the
STFT are not given in the literature. We do this in
Theorem 1.1. $(iv)$ and Theorem 1.2 $(iv)$. Moreover,
Theorem 1.1 $(iii)$ as well as Theorem 1.2 $(ii)$--$(iii)$
give new characterisations of both wave front sets.
The purpose of this article is to fill in these gaps.
Note that the $L^q$, $q\in[1,\infty)$, versions, so
called $\mathcal F L^q_\omega$ weighted wave
fronts \cite{PTT} (equal to the Sobolev one for $q=2$)
can be treated in the same way as in the case $q=2$.

\section{The main results}

We begin by recalling the definition of H\"ormander \cite{hor11} (cf. \cite[Definition 8.1.2, p. 254]{hor}) for the wave front set $WF(f)\subseteq \RR^d\times(\RR^d\backslash\{0\})$ of $f\in\DD'(\RR^d)$:\\
\indent $(x_0,\xi_0)\in\RR^d\times(\RR^d\backslash\{0\})$ does not belong to $WF(f)$ if there exists $\chi\in\DD(\RR^d)$ with $\chi(x_0)\neq 0$ and a cone neighbourhood $\Gamma$ of $\xi_0$ such that for every $n\in\NN$ there exists $C_{n,\chi}>0$ such that \beq\label{sftkrn137}
|\mathcal{F}(\chi f)(\xi)|\leq C_{n,\chi}(1+|\xi|)^{-n},\,\, \forall \xi\in\Gamma.
\eeq

\begin{theorem}
Let $f\in\DD'(\RR^d)$ and $(x_0,\xi_0)\in\RR^d\times(\RR^d\backslash\{0\})$. The following conditions are equivalent.
\begin{itemize}
\item[$(i)$] $(x_0,\xi_0)\not\in WF(f)$.
\item[$(ii)$] There exist a compact neighbourhood $K$ of $x_0$ and a cone neighbourhood $\Gamma$ of $\xi_0$ such that for every $n\in\NN$ and $\chi\in\DD_K$\footnote{as usual, $\DD_K$ stands for the Fr\'echet space of all smooth functions supported by $K$} there exists $C_{n,\chi}>0$ such that (\ref{sftkrn137}) is valid
\item[$(iii)$] There exist a compact neighbourhood $K$ of $x_0$ and a cone neighbourhood $\Gamma$ of $\xi_0$ such that for each $n\in\NN$ there exist $C_n>0$ and $k_n\in\NN$ such that
    \beqs
    |V_{\chi}(f)(x,\xi)|\leq C_n\sup_{|\alpha|\leq k_n}\|D^{\alpha}\chi\|_{L^{\infty}(\RR^d)}(1+|\xi|)^{-n},\,\, \forall x\in K,\, \forall \xi\in\Gamma,\, \forall \chi\in \DD_{K-\{x_0\}},
    \eeqs
    where $K-\{x_0\}=\{y\in\RR^d|\, y+x_0\in K\}$.
\item[$(iv)$] There exist a compact neighbourhood $K$ of $x_0$, a cone neighbourhood $\Gamma$ of $\xi_0$ and $\chi\in\DD(\RR^d)$, with $\chi(0)\neq 0$ such that for each $n\in\NN$ there exists $C_{n,\chi}>0$ such that
    \beqs
    |V_{\chi}(f)(x,\xi)|\leq C_{n,\chi}(1+|\xi|)^{-n},\,\, \forall x\in K,\, \forall \xi\in\Gamma.
    \eeqs
\end{itemize}
\end{theorem}

\begin{proof} $(i)\Rightarrow (ii)$
The fact that $(x_0,\xi_0)\not\in WF(f)$ implies the existence of $\chi\in\DD(\RR^d)$ with $\chi(x_0)\neq 0$ and a cone neighbourhood $\Gamma'$ of $\xi_0$ for which (\ref{sftkrn137}) is valid for $\xi\in\Gamma'$. There exists a compact neighbourhood $K$ of $x_0$ where $\chi$ never vanishes. Fix a cone neighbourhood $\Gamma$ of $\xi_0$ such that $\overline{\Gamma}\subseteq \Gamma'\cup\{0\}$.
 By employing exactly the same technique as in the proof of \cite[Lemma 8.1.1, p. 253]{hor} we obtain that for each $n\in\NN$ and $\psi\in\DD_K$ we have $|\mathcal{F}(\psi\chi f)(\xi)|\leq C_{n,\psi,\chi}(1+|\xi|)^{-n}$, $\forall \xi\in\ \Gamma$. Then $(ii)$  immediately follows from this since $\psi f=(\psi/\chi)\chi f$ where $\psi/\chi\in\DD_K$.

$(ii)\Rightarrow (iii)$ By $(ii)$, there exists a compact neighbourhood $K_1$ of $x_0$ and a cone neighbourhood $\Gamma$ of $\xi_0$ such that for every $n\in\NN$ and $\chi\in\DD_{K_1}$ there exists $C_{n,\chi}>0$ such that (\ref{sftkrn137}) is valid. Of course, we can assume $K_1=\overline{B_r(x_0)}$, for some $r>0$, where $B_r(x_0)$ stands for the open ball with centre $x_0$ and radius $r$. Notice that (\ref{sftkrn137}) implies that for each $n\in\NN$, the set $H_n=\{(1+|\xi|)^n e^{-i\xi\cdot} f|\, \xi\in \Gamma\}$ is weakly bounded in $\DD'_{K_1}$ and hence equicontinuous as $\DD_{K_1}$ is barrelled. Put $K=\overline{B_{r/2}(x_0)}$. For each $\chi\in\DD_{K-\{x_0\}}$ and $x\in K$ the function $t\mapsto \chi(t-x)$ is in $\DD_{K_1}$ and the equicontinuity of $H_n$ implies the existence of $C_n>0$ and $k_n\in\NN$ such that
\beqs
|\langle e^{-i\xi\cdot} f,\overline{\chi(\cdot-x)}\rangle|&\leq& C_n(1+|\xi|)^{-n}\sup_{|\alpha|\leq k_n}\sup_{t\in K_1}|D^{\alpha}\chi(t-x)|\\
&=&C_n\sup_{|\alpha|\leq k_n}\|D^{\alpha}\chi\|_{L^{\infty}(\RR^d)}(1+|\xi|)^{-n},\,\, \forall \xi\in\Gamma,\, \forall x\in K,
\eeqs
which implies the validity of $(iii)$. Notice that $(iii)\Rightarrow (iv)$ is trivial and $(iv)\Rightarrow (i)$ follows immediately by specialising the estimate in $(iv)$ for $x=x_0$.
\end{proof}

We can also give similar characterisation of the Sobolev wave front. Before we give the result we recall its definition (see \cite[Definition 8.2.5, p. 188; Proposition 8.2.6, p. 189]{hor0}):\\
\indent Let $f\in\DD'(\RR^d)$ and $s\in\RR$. We say that $(x_0,\xi_0)\in\RR^d\times(\RR^d\backslash\{0\})$ does not belong to $WF_{H^s}(f)$ if there exist $\chi\in\DD(\RR^d)$ with $\chi(x_0)\neq 0$ and a cone neighbourhood $\Gamma$ of $\xi_0$ such that $\|\langle \cdot\rangle^s\mathcal{F}(\chi f)\|_{L^2(\Gamma)}<\infty$.\footnote{as usual, $\langle \xi\rangle$ stands for $(1+|\xi|^2)^{1/2}$}

\begin{theorem}
Let $f\in\DD'(\RR^d)$, $(x_0,\xi_0)\in\RR^d\times(\RR^d\backslash\{0\})$ and $s\in\RR$. The following conditions are equivalent.
\begin{itemize}
\item[$(i)$] $(x_0,\xi_0)\not\in WF_{H^s}(f)$.
\item[$(ii)$] There exist a compact neighbourhood $K$ of $x_0$ and a cone neighbourhood $\Gamma$ of $\xi_0$ such that the mapping $\chi\mapsto \langle \cdot\rangle^s\mathcal{F}(\chi f)$, $\DD_K\rightarrow L^2(\Gamma)$, is well-defined and continuous.
\item[$(iii)$] There exist a compact neighbourhood $K$ of $x_0$, a cone neighbourhood $\Gamma$ of $\xi_0$ and $C>0$ and $k\in\NN$ such that
    \beqs
    \sup_{x\in K}\|\langle \cdot\rangle^s V_{\chi}f(x,\cdot)\|_{L^2(\Gamma)}\leq C\sup_{|\alpha|\leq k}\|D^{\alpha}\chi\|_{L^{\infty}(\RR^d)},\,\, \forall \chi\in\DD_{K-\{x_0\}}.
    \eeqs
\item[$(iv)$] There exist a compact neighbourhood $K$ of $x_0$, a cone neighbourhood $\Gamma$ of $\xi_0$ and $\chi\in\DD(\RR^d)$ with $\chi(0)\neq0$ such that $\sup_{x\in K}\|\langle \cdot\rangle^s V_{\chi}f(x,\cdot)\|_{L^2(\Gamma)}<\infty$.
\end{itemize}
\end{theorem}

\begin{proof} $(i)\Rightarrow (ii)$ There exist a cone neighbourhood $\Gamma'$ of $\xi_0$ and $\chi\in\DD(\RR^d)$ with $\chi(x_0)\neq 0$ such that $C_{\chi}=\|\langle \cdot\rangle^s\mathcal{F}(\chi f)\|_{L^2(\Gamma')}<\infty$. There exists a compact neighbourhood $K$ of $x_0$ where $\chi$ never vanishes and there are $C_1,m\geq 1$ such that $|\mathcal{F}(\chi f)(\xi)|\leq C_1\langle \xi\rangle^m$, $\forall \xi\in\RR^d$. Let $\Gamma$ be a cone neighbourhood of $\xi_0$ such that $\overline{\Gamma}\subseteq \Gamma'\cup\{0\}$. One can find $0<c<1$ such that
\beq\label{vsrlkt135}
\{\eta\in\RR^d|\, \exists \xi\in \Gamma\, \mbox{such that}\, |\xi-\eta|\leq c|\xi|\}\subseteq \Gamma'.
\eeq
For $\psi\in\DD_K$, we have $\mathcal{F}(\psi\chi f)=(2\pi)^{-d}\mathcal{F}\psi*\mathcal{F}(\chi f)$ and the Minkowski integral inequality yields
\beqs
\|\langle \cdot\rangle^s \mathcal{F}(\psi\chi f)\|_{L^2(\Gamma)}&\leq& \frac{1}{(2\pi)^d} \int_{\RR^d_{\eta}}\left(\int_{\Gamma}\langle \xi\rangle^{2s}|\mathcal{F}\psi(\eta)|^2|\mathcal{F}(\chi f)(\xi-\eta)|^2d\xi\right)^{1/2} d\eta\\
&\leq& I_1/(2\pi)^d+I_2/(2\pi)^d,
\eeqs
where
\beqs
I_1&=&\int_{\RR^d_{\eta}}|\mathcal{F}\psi(\eta)|\left(\int_{\substack{|\xi|\geq|\eta|/c\\ \xi\in\Gamma}}\langle \xi\rangle^{2s}|\mathcal{F}(\chi f)(\xi-\eta)|^2d\xi\right)^{1/2} d\eta,\\
I_2&=&\int_{\RR^d_{\eta}}|\mathcal{F}\psi(\eta)|\left(\int_{\substack{|\xi|<|\eta|/c\\ \xi\in\Gamma}}\langle \xi\rangle^{2s}|\mathcal{F}(\chi f)(\xi-\eta)|^2d\xi\right)^{1/2} d\eta.
\eeqs
First we assume $s\geq 0$. A change of variables in the inner integral in $I_1$ gives
\beqs
I_1&=&\int_{\RR^d_{\eta}}|\mathcal{F}\psi(\eta) |\left(\int_{\substack{|\xi+\eta|\geq|\eta|/c\\ \xi\in\Gamma-\{\eta\}}}\langle \xi+\eta\rangle^{2s}|\mathcal{F}(\chi f)(\xi)|^2d\xi\right)^{1/2} d\eta\\
&\leq&(1-c)^{-s}\int_{\RR^d_{\eta}}|\mathcal{F}\psi(\eta)|\left(\int_{\Gamma'}\langle \xi\rangle^{2s}|\mathcal{F}(\chi f)(\xi)|^2d\xi\right)^{1/2} d\eta= \frac{C_{\chi}\|\mathcal{F}\psi\|_{L^1(\RR^d)}}{(1-c)^s},
\eeqs
where, in the inequality we have used $\{\xi\in\Gamma-\{\eta\}|\, |\xi+\eta|\geq|\eta|/c\}\subseteq \Gamma'$ which easily follows from (\ref{vsrlkt135}). For $I_2$ we have
\beqs
I_2&\leq& C_1\int_{\RR^d_{\eta}}|\mathcal{F}\psi(\eta)|\left(\int_{\substack{|\xi|<|\eta|/c\\ \xi\in\Gamma}}\langle \xi\rangle^{2s}\langle \xi-\eta\rangle^{2m} d\xi\right)^{1/2} d\eta\\
&\leq& C_1(1+c^{-1})^mc^{-s-d-1}\|\langle \cdot\rangle^{-d-1}\|_{L^2(\RR^d)}\|\langle \cdot\rangle^{m+s+d+1}\mathcal{F}\psi\|^2_{L^1(\RR^d)}.
\eeqs
Combining these estimates together we conclude that there exists $C=C(\chi)>0$ such that $\|\langle \cdot\rangle^s\mathcal{F}(\psi\chi f)\|_{L^2(\Gamma)}\leq C\|\langle \cdot\rangle^{s+m+d+1}\mathcal{F}\psi\|_{L^1(\RR^d)}$, $\forall \psi\in\DD_K$. Since for $\psi\in\DD_K$, we have $\psi f=(\psi/\chi)\chi f$ with $\psi/\chi\in\DD_K$, the claim in $(ii)$ easily follows from this. The case when $s<0$ is similar and we omit it.\\
\indent $(ii)\Rightarrow (iii)$ Let $K_1$ be a compact neighbourhood of $x_0$ and $\Gamma$ a cone neighbourhood of $\xi_0$ such that the mapping $\chi\mapsto \langle \cdot\rangle^s\mathcal{F}(\chi f)$, $\DD_{K_1}\rightarrow L^2(\Gamma)$, is well-defined and continuous; of course, without losing generality, we can assume $K_1=\overline{B_r(x_0)}$, for some $r>0$. There exist $C>0$ and $k\in\NN$ such that
\beqs
\|\langle \cdot \rangle^s\mathcal{F}(\chi f)\|_{L^2(\Gamma)}\leq C\sup_{|\alpha|\leq k}\|D^{\alpha}\chi\|_{L^{\infty}(K_1)},\,\, \forall \chi\in \DD_{K_1}.
\eeqs
Put $K=\overline{B_{r/2}(x_0)}$. For $\chi\in \DD_{K-\{x_0\}}$ and $x\in K$, the function $\chi_x:t\mapsto \overline{\chi(t-x)}$ belongs to $\DD_{K_1}$ and, as $\mathcal{F}(\chi_x f)(\xi)=V_{\chi}f(x,\xi)$, we have
\beqs
\sup_{x\in K}\|\langle\cdot\rangle^s V_{\chi}f(x,\cdot)\|_{L^2(\Gamma)}\leq C\sup_{x\in K}\sup_{|\alpha|\leq k}\|D^{\alpha}\chi_x\|_{L^{\infty}(K_1)}=C\sup_{|\alpha|\leq k}\|D^{\alpha}\chi\|_{L^{\infty}(\RR^d)}.
\eeqs
The implications $(iii)\Rightarrow (iv)$ is trivial and $(iv)\Rightarrow (i)$ follows immediately by specialising $x=x_0$.
\end{proof}

\end{document}